\documentclass[12pt]{article} \textheight 8.8 true in \textwidth
6.2 true in

\usepackage[T2A]{fontenc}
\usepackage[cp1251]{inputenc}
\usepackage[english]{babel}
\usepackage{amsfonts}
\usepackage{amssymb,amsmath,amstext}
\usepackage{amscd, amsthm, latexsym}
\usepackage{color}
\usepackage[final]{graphicx}
\usepackage{epsfig}
\usepackage{euscript}
\usepackage{srcltx,epsf}
\usepackage{wrapfig}
\usepackage{mathrsfs}
\usepackage{ifthen}
\usepackage{cleveref}

\def\hide#1{}
\def\old#1{}

\def\oop#1{}
\def\gap#1{}

\let\eps\varepsilon

\theoremstyle{plain}
\newtheorem{theorem}{Theorem}
\newtheorem{proposition}{Proposition}
\newtheorem*{theorem*}{Theorem}
\newtheorem{lemma}{Lemma}
\newtheorem{remark}{Remark}
\newtheorem{corollary}{Corollary}

\def \ZZ  {\Bbb Z}

\begin{document}
\newcounter{figcounter}
\setcounter{figcounter}{0}
\addtocounter{figcounter}{1}

\title{
On the coverings of  Euclidian manifolds $\mathcal{B}_1$ and $\mathcal{B}_2$} 
\author{
G.~Chelnokov\\
{\small\tt grishabenruven@yandex.ru }
\\ [2ex]
M.~Deryagina\thanks{This work was supported by the Russian Foundation for Basic Research (grant 13-01-00513).}\\
{\small\em Sobolev Institute of Mathematics, Russia}\\
{\small\em Moscow State University of Technologies and} \\
{\small\em Management named after K.G. Razumovskiy, Russia}\\
{\small\tt madinaz@rambler.ru }
\\ [2ex]
A.~Mednykh\thanks{This work was supported by the Laboratory of
Quantum Topology, Chelyabinsk State University,
  under contract no. 14.Z50.31.0020 with the Ministry of Education
  and Science of the Russian Federation, and the Russian Foundation for Basic Research (grant 15-01-07906).}\\
{\small\em Sobolev Institute of Mathematics,} \\
{ \small\em Novosibirsk, Russia}\\
{\small\em Novosibirsk State University,}\\
{\small\em Novosibirsk, Russia}\\
{\small\em Chelyabinsk State University,}\\
{\small\em Chelyabinsk, Russia}\\
 {\small\tt mednykh@math.nsc.ru}
}
\date{}
\maketitle

\begin{abstract}
  There are only 10 Euclidean three dimensional forms, that is, compact locally Euclidean 3-manifold without
boundary. Six of them, $\mathcal{G}_{1}$, $\mathcal{G}_{2}$, $\mathcal{G}_{3}$,
$\mathcal{G}_{4}$, $\mathcal{G}_{5}$, $\mathcal{G}_{6},$  are orientable and the other four, $\mathcal{B}_{1}$, $\mathcal{B}_{2}$,
$\mathcal{B}_{3}$, $\mathcal{B}_{4},$ are non-orientable manifolds. Following J.~H.~Conway  and J.~P.~Rossetti, we call them platycosms.  In the present paper,  a new algebraic  method  is given to classify  and enumerate $n$-fold coverings over platycosms.
We decribe all types of $n$-fold coverings over $\mathcal{B}_{1}$ and over
$\mathcal{B}_{2}$, and calculate the numbers of non-equivalent
coverings of each type. Recall that the manifolds $\mathcal{B}_{1}$ and $\mathcal{B}_{2}$ are uniquely determined among the others non-orientable forms by their homology groups $H_1(\mathcal{B}_{1})=\ZZ_2\times\ZZ^2$ and
$H_1(\mathcal{B}_{2})=\ZZ^2$.


\end{abstract}

\section*{Introduction}
 Let $\mathcal{N}$, $\mathcal{N'}$ and  $\mathcal{M}$  be connected manifolds. Two coverings  \mbox{$\rho : \mathcal{N} \longrightarrow \mathcal{M}$} and $\rho' : \mathcal{N'} \longrightarrow \mathcal{M}$  are  \emph{equivalent} if  there exists a homeomorphism  $\eta : \mathcal{N} \longrightarrow \mathcal{N'}$ such that $\rho = \rho' \circ \eta$.

From general theory of covering spaces, it follows that any $n$-fold
covering over $\mathcal{M}$ corresponds to the unique subgroup of
index $n$ in the fundamental group $\pi_{1}(\mathcal{M})$. Two
coverings are equivalent if and only if
 the corresponding subgroups are conjugate in $\pi_{1}(\mathcal{M})$ (see \cite{Hatcher} p.~67). The equivalence classes of $n$-fold covering of $N$ are in one-to-one correspondence with the conjugacy classes of subgroups of index $n$ in the fundamental group $\pi_1(N).$

The number of subgroups of a given index in the free group $F_r$ was determined by M. Hall \cite{Ha49}. An explicit formula for the number of conjugacy classes of subgroups of a given index in $F_r$ was obtained by V.~A.~Liskovets \cite{Li71}. Both problems for the fundamental group $\Gamma_g$ of a closed orientable surface of genus $g$ were completely resolved  in \cite{Med78} and \cite{Med79} respectively. For the fundamental group $\Phi_p$ of a closed non-orientable surface of genus $p$ they were resolved in \cite{MP86}. See paper \cite{Medn} for a short proof of the above mentioned results.
 In the paper \cite{LisM} the results were extended    to some 3-dimensional manifolds which are circle bundles over a surface.

  Following Conway-Rossetti (\cite{Conway}), we use the term platycosm
("flat universe") for a compact locally Euclidean 3-manifold without
boundary. In the present paper we suggest a new algebraic  method  to classify   $n$-fold coverings over amphicosms and enumerate   them. It is important to emphasize that numerical methods to solve these problems was developed by the Bilbao group \cite{babaika}.

Platycosms   are the simplest alternative universes for us
to think of living in \cite{LSW}. (Jeff Weeks created the program,
which allows to "fly" through platycosms and other spaces.)
If you lived in a small enough platycosm, you would appear to be
surrounded by images of yourself which can be arranged in one of ten
essentially different ways. (See \cite{Conway}).


There are six orientable platicosms, denoted in Wolfs notation
$\mathcal{G}_{1}$, $\mathcal{G}_{2}$, $\mathcal{G}_{3}$,
$\mathcal{G}_{4}$, $\mathcal{G}_{5}$, $\mathcal{G}_{6}$ and four
non-orientable manifolds $\mathcal{B}_{1}$, $\mathcal{B}_{2}$,
$\mathcal{B}_{3}$, $\mathcal{B}_{4}$. The aim of this paper is to
classify types of $n$-fold coverings over $\mathcal{B}_{1}$ and over
$\mathcal{B}_{2}$, and calculate the numbers of non-equivalent
coverings of each type. We classify all types of subgroups in the
fundamental groups $\mathcal{B}_{1}$ and $\mathcal{B}_{2}$
respectively, and   calculate the numbers
 of conjugated classes of  each type of subgroups for index $n$.

\subsection*{Notations}
During this paper we will use the following notations:
$s_G(n)$ is the total number of subgroups of index $n$ in the group
$G$, $s_{H,G}(n)$ is the number of subgroups of index $n$ in the
group $G$, isomorphic to the group $H$. The same way $c_G(n)$ is the
total number of conjugancy classes of subgroups of index $n$ in the
group $G$, $c_{H,G}(n)$ is the number conjugancy classes of
subgroups of index $n$ in the group $G$, isomorphic to the group
$H$.

$$\sigma_0(n) = \displaystyle \sum_{k \mid n}1 \quad \mbox{if $n$ is
natural, $\sigma_0(n)=0$ otherwise},$$

$$\sigma_1(n) = \displaystyle \sum_{k \mid n} k \quad \mbox{if $n$ is
natural, $\sigma_1(n)=0$ otherwise},$$
$$\sigma_2(n) = \displaystyle \sum_{k \mid n} \sigma_1(k) \quad
\mbox{if $n$ is natural, $\sigma_2(n)=0$ otherwise}.
$$
The epimorphism $\phi: \pi_1(\mathcal{B}_1) \to \ZZ^2$ is determined
in \Cref{prop2}, the epimorphism $\psi: \pi_1(\mathcal{B}_2) \to
\ZZ^2$ is determined in \Cref{prop2-2}.

The invariant of a subgroup $\Delta \leqslant \pi_1(\mathcal{B}_1)$,
$l(\Delta)$ is determined in \Cref{prop3}, its analog for $\Delta
\leqslant \pi_1(\mathcal{B}_2)$ is determined in \cref{prop2-3}.

The invariants $\rho(\Delta)$ and $\eps(\Delta)$, applied to
non-abelian subgroup $\Delta\leqslant \pi_1(\mathcal{B}_1)$ only,
are determined after \Cref{prop11}, analogous ones for
$\Delta\leqslant \pi_1(\mathcal{B}_2)$ are determined after
\Cref{prop2-11}.

In this paper we widely use the the summing $\displaystyle\sum_{2k
\mid n}$. We consider it equals zero if $n$ is odd since in this
case the sum is taken over the empty set of terms.

\bigskip

One can find the correspondence between Wolf's and Conway-Rossetti's
notations of flat 3-manifold and its fundamental groups in Table~1.
We use Wolf's notation.

\bigskip

\begin{tabular}{|l|l|l|l|l|}
\hline name&\begin{tabular}{c}Conway-\\Rosetti\end{tabular}&other
names&Wolf&\begin{tabular}{c}fund.group\\(internatl.\\no
name)\end{tabular} \\ \cline{1-5} \hline
 torocosm&$c_1$&3-torus&$\mathcal{G}_{1}$&\;\;\;\;\;1.$P1$\\
\hline
dicosm&$c_2$&half turn space&$\mathcal{G}_{2}$&\;\;\;\;\;4.$P2_1$ \\
\hline tricosm&$c_3$&one-third turn
space&$\mathcal{G}_{3}$&\begin{tabular}{c}
144.$P3_1$ \\
145.$P3_2$
\end{tabular}\\
\hline tetracosm&$c_4$&quarter turn
space&$\mathcal{G}_{4}$&\begin{tabular}{c} \;\;76.$P4_1$\\
\;\;78.$P4_3$
\end{tabular} \\
\hline
hexacosm&$c_6$&one-sixth turn space&$\mathcal{G}_{5}$&\begin{tabular}{c}169.$P6_1$\\170.$P6_5$\end{tabular}\\
\hline
didicosm&$c_{22}$&\begin{tabular}{c}Hantzsche-Wendt\\space\end{tabular}&$\mathcal{G}_{6}$&\;\;\;\;19.$P2_12_12_1$\\
\hline
first amphicosm&$+a_1$&\begin{tabular}{c}Klein bottle times\\circle\end{tabular}&$\mathcal{B}_{1}$&\;\;\;\;\;7.$Pc$\\
\hline
second amphicosm&$-a_1$& &$\mathcal{B}_{2}$&\;\;\;\;\;9.$Cc$\\
\hline
 first amphidicosm  & $+a_2$  &  & $\mathcal{B}_{3}$  & \;\;\;29.$Pca2_1$ \\
\hline
 second amphidicosm  & $-a_2$  &  & $\mathcal{B}_{4}$  & \;\;\;33.$Pa2_1$ \\
\hline
\end{tabular}
\begin{center}

\small{Table~1}
\end{center}

%

\section{The brief overview of achieved results}
Since the problem of enumeration of $n$-fold coverings reduces to
the problem of enumeration of conjugacy classes of some subgroups,
it is natural to expect that the enumeration of subgroups without
respect of conjugacy would be helpful. The next theorem provides the
complete solution of this problem. For brevity we use notation $Pc$
for $\pi_{1}(\mathcal{B}_{1})$ and $Cc$ for
$\pi_{1}(\mathcal{B}_{2})$, see Table~1.

\begin{theorem}\label{th1}
Every subgroup $\Delta$ of finite index $n$ in
$\pi_{1}(\mathcal{B}_{1})$ is isomorphic to either $\ZZ^3$, or
$\pi_{1}(\mathcal{B}_{1})$, or $\pi_{1}(\mathcal{B}_{2})$, and
$$
s_{\ZZ^3, \pi_1(\mathcal{B}_{1})}(n) =  \sum_{2l \mid n}
\sigma_1(l)l,\leqno (i)
$$
$$
s_{\pi_{1}(\mathcal{B}_{1}), \pi_1(\mathcal{B}_{1})}(n) = \sum_{l
\mid n} \sigma_1(\frac{n}{l})l-\sum_{2l \mid
n}\sigma_1(\frac{n}{2l})l, \leqno (ii)
$$
$$
s_{\pi_{1}(\mathcal{B}_{2}), \pi_1(\mathcal{B}_{1})}(n) = \sum_{2l
\mid n} 2l\sigma_1(\frac{n}{2l})-\sum_{4l \mid
n}2l\sigma_1(\frac{n}{4l}). \leqno (iii)
$$

\end{theorem}
To prove this theorem we need the following propositions. To prove
this theorem we need the following propositions. \Cref{prop2}
presents a canonical form of elements in $\pi_1(\mathcal{B}_1)$ and
introduce an invariant $\phi(\Delta)$. \Cref{prop4} provides a
necessary and sufficient condition that the subgroup
$\Delta\leqslant \pi_1(\mathcal{B}_1)$ is abelian in terms of
$\phi(\Delta)$. This proposition also provides the fact that all
abelian subgroups of finite index in $\pi_1(\mathcal{B}_1)$ belongs
to one largest abelian subgroup (this is certainly false for abelian
subgroups of infinite index). Finally, \Cref{prop12} show that
4-plet of introduced invariants
$((l(\Delta),\phi(\Delta),\rho(\Delta),\eps(\Delta))$ determine a
non-abelian subgroup uniquely and \Cref{prop13} describes the type
of subgroup in terms of $\eps(\Delta)$.

Theorem 2 provides the total number of conjugacy classes of
subgroups of index $n$ in $\pi_{1}(\mathcal{B}_{1})$, thus the total
number of non-equivalent $n$-fold coverings of $\mathcal{B}_{1}$.
\begin{theorem}\label{th1.2}
The total number of non-equivalent $n$-fold coverings over
$\mathcal{B}_1$ is
\begin{displaymath}
 c_{\pi_{1}(\mathcal{B}_{1})}(n)= \frac{1}{n}\sum_{\substack{l | n \\ l m = n}}(\varphi_{3}(l)\sum_{2k \mid m} \sigma_1(k)k +  \sum_{\substack{ d | l}}\,\mu \left( \frac{l}{d}\right) (2,d) d^{2}\sum_{k \mid m} (\sigma_1(\frac{m}{k})-\sigma_1(\frac{m}{2k}))k +
 \end{displaymath}
 \begin{displaymath}
 +\varphi_{2}(l)\sum_{2k \mid m}
(\sigma_1(\frac{m}{2k})-\sigma_1(\frac{m}{4k}))2k)\,,
 \end{displaymath}
 where  $\mu (k)$ is the M\"{o}bius function, $\varphi_{3} (l)$ and  $\varphi_{2} (l)$ are Jordan totient functions,$(2,d)$ is a greater common divisor of numbers  $2$ and $d$.
\end{theorem}
The next theorem provides the number of conjugacy classes of
subgroups of index $n$ in $\pi_{1}(\mathcal{B}_{1})$ with respect of
the isomorphism type of a subgroup, thus the number of $n$-fold
coverings of $\mathcal{B}_{1}$ with respect of the isomorphism type
of a cover.
\begin{theorem}\label{th1.3}
Let $\mathcal{N} \to \mathcal{B}_{1}$ be an $n$-fold covering over
$\mathcal{B}_{1}$. If $n$ is odd then $\mathcal{N}$ is homeomorphic
to $\mathcal{B}_{1}$. If $n$ is even then $\mathcal{N}$ is
homeomorphic to $\mathcal{G}_{1}$ or $\mathcal{B}_{1}$ or
$\mathcal{B}_{2}$. The corresponding numbers of nonequivalent
coverings are given by the following formulas:
$$
c_{\ZZ^3,\pi_{1}(\mathcal{B}_{1})}(n) = \sum_{2l \mid n} \sum_{m
\mid \frac{n}{2l}} \Big(l^2 + \frac{5}{2} +\frac{3}{2}(-1)^l\Big)m,
\leqno (i)
$$
$$
c_{\pi_{1}(\mathcal{B}_{1}),\pi_{1}(\mathcal{B}_{1})}(n)=\sum_{l
\mid n}
\Big(\frac{3}{2}+\frac{1}{2}(-1)^l\Big)\sigma_1(\frac{n}{l})-\sum_{2l
\mid
n}\Big(\frac{3}{2}+\frac{1}{2}(-1)^l\Big)\sigma_1(\frac{n}{2l}),
\leqno (ii)
$$
$$
c_{\pi_{1}(\mathcal{B}_{2}),\pi_{1}(\mathcal{B}_{1})}(n) =2\Big(
\sum_{2l \mid n} \sigma_1(\frac{n}{2l})-\sum_{4l \mid
n}\sigma_1(\frac{n}{4l})\Big).\leqno (iii)
$$
\end{theorem}
Despite \Cref{th1.3} covers \Cref{th1.2} we present both, since the
proof of \Cref{th1.2} can be also obtained from other
considerations, see \Cref{lulz2}.

The results towards the manifold $\mathcal{B}_2$ follows much the
same way. \Cref{th-Second-1}, \Cref{th-Second-1.2} and
\Cref{th-Second-1.3} are similar to \Cref{th1}, \Cref{th1.2} and
\Cref{th1.3} respectively.

\begin{theorem}\label{th-Second-1}
Every subgroup $\Delta$ of finite index $n$ in
$\pi_{1}(\mathcal{B}_{2})$ is isomorphic to either $\ZZ^3$, or
$\pi_1(\mathcal{B}_{2})$, or $\pi_1(\mathcal{B}_{1})$, and
$$
s_{\ZZ^3, \pi_1(\mathcal{B}_{2})}(n) =  \sum_{2l \mid n}
\sigma_1(l)l,\leqno (i)
$$
$$
s_{\pi_{1}(\mathcal{B}_{2}),\pi_{1}(\mathcal{B}_{2})}=\left\{
\begin{aligned}
\sum_{4k \mid n}2k\big(\sigma_1(\frac{n}{2k})-\sigma_1(\frac{n}{4k})\big) \quad\text{if } n \,\,\text {is even} \\
\sum_{k \mid n} \sigma_1(\frac{n}{k})k \quad \text{if } n \,\,\text {is odd} \\
\end{aligned} \right. \leqno (ii)
$$
$$
s_{\pi_{1}(\mathcal{B}_{1}),\pi_{1}(\mathcal{B}_{2})}=\sum_{2k \mid
n} 2k\sigma_1(\frac{n}{2k})-\sum_{4k \mid
n}2k\sigma_1(\frac{n}{4k}). \leqno (iii)
$$
\end{theorem}

\begin{theorem}\label{th-Second-1.2}
The total number of non-equivalent $n$-fold coverings over
$\mathcal{B}_2$ is
\begin{displaymath}
 c_{\pi_{1}(\mathcal{B}_{2})}= \frac{1}{n}\sum_{\substack{l | n \\ l m = n}}(\varphi_{3}(l)\sum_{2k \mid m} \sigma_1(k)k +   s_{\pi_{1}(\mathcal{B}_{1}),\pi_{1}(\mathcal{B}_{2})}(m)\sum_{\substack{ d | l}}\,\mu \left( \frac{l}{d}\right) (2,d) d^{2}  +\varphi_{2}(l)s_{\pi_{1}(\mathcal{B}_{2}),\pi_{1}(\mathcal{B}_{2})} (m))\,,
 \end{displaymath}
 where  $\mu (t)$ is the M\"{o}bius function, $\varphi_{3} (l)$ and  $\varphi_{2} (l)$ are Jordan totient functions,$(2,d)$ is a greater common divisor of numbers  $2$ and $d$.
\end{theorem}

\begin{theorem}\label{th-Second-1.3}
Let $\mathcal{N} \to \mathcal{B}_{2}$ be an $n$-fold covering over
$\mathcal{B}_{2}$. If $n$ is odd then $\mathcal{N}$ is homeomorphic
to $\mathcal{B}_{2}$. If $n$ is even then $\mathcal{N}$ is
homeomorphic to $\mathcal{G}_{1}$ or $\mathcal{B}_{1}$ or
$\mathcal{B}_{2}$. The corresponding numbers of nonequivalent
coverings are given by the following formulas:
$$
c_{\ZZ^3,\pi_1(\mathcal{B}_2)}(n)= \frac{1}{2}\sum_{2l \mid n}
\sum_{m \mid \frac{n}{2l}} \Big(l^2 +
\frac{3}{2}+\frac{1}{2}(-1)^{l}+(-1)^{\frac{n}{2lm}} +
(-1)^{l+\frac{n}{2lm}}\Big)m, \leqno (i)
$$
$$
c_{\pi_1(\mathcal{B}_2),\pi_1(\mathcal{B}_2)}(n)=\left\{
\begin{aligned}
\sum_{4k \mid n}(\sigma_1(\frac{n}{2k})-\sigma_1(\frac{n}{4k})) \quad\text{if } n \,\,\text {is even} \\
\sum_{l \mid n} \sigma_1(\frac{n}{l}) \quad \text{if } n \,\,\text {is odd} \\
\end{aligned} \right. \leqno (ii)
$$
$$
c_{\pi_1(\mathcal{B}_1),\pi_1(\mathcal{B}_2)}(n)= 2\Big(\sum_{2k
\mid n}\sigma_1(\frac{n}{2k}) - \sum_{4k \mid
n}\sigma_1(\frac{n}{4k})\Big). \leqno (iii)
$$
\end{theorem}

\section{Preliminaries}

Consider manifold, referred as $\mathcal{B}_{1}$ and
$\mathcal{B}_{2}$ in \cite{Wolf}, also as $+a_{1}$ and $-a_{1}$ in
\cite{Conway}.

 Let us remark that $\mathcal{B}_{1}$ can be considered as a Seifert fiber space.  $\mathcal{B}_{1}$ is the  trivial $S^{1}$-bundle  over Klein bottle $\mathcal{K}$,
 so  $\mathcal{B}_{1}=\mathcal{K}\times S^{1}$.
 Its fundamental group $\pi_{1}(\mathcal{B}_{1})$ is
 \begin{equation}\label{fundplus1}
 \pi_{1}(\mathcal{B}_{1})=\pi_{1}(\mathcal{K})\times\pi_{1}(S^{1})=\Lambda \times \ZZ,
 \end{equation}
  where $\Lambda$ is a fundamental group of Klein bottle.


From \cite{Conway} the fundamental group $\pi_{1}(\mathcal{B}_{1})$,
can be represented in the form
\begin{equation}\label{fundplus2}
 \pi_{1}(\mathcal{B}_{1})=\langle W, X, Z: X^{-1}ZX = W^{-1}ZW=Z^{-1},
 W^{-1}Z^{-1}WZ=1
 \rangle .
 \end{equation}
It is also a crystallographic group $Pc$.

The fundamental group $\pi_{1}(\mathcal{B}_{2})$, can be represented
in the form
\begin{equation}\label{fundminus2}
 \pi_{1}(\mathcal{B}_{2})=\langle W', X', Z': X'^{-1}Z'X' = W'^{-1}Z'W'=Z'^{-1},  W'^{-1}Z'^{-1}W'Z'=Z'
 \rangle .
 \end{equation}
It is also a crystallographic group $Cc$.


For our convenience we replace $X=a$, $W=ca$, $Z=b$ and obtain
\begin{equation}\label{fundplus3}
 \pi_{1}(\mathcal{B}_{1})=\langle a,b,c: cac^{-1}a^{-1}=cbc^{-1}b^{-1}=1,  aba^{-1}b=1 \rangle.
 \end{equation}

 Substituting $X'=\alpha$, $Z'=\beta^{-1}$ and $W'=\alpha\gamma$, we get

\begin{equation}\label{fundminus3}
 \pi_{1}(\mathcal{B}_{2})=\langle \alpha, \beta, \gamma: \gamma\beta\gamma^{-1}=\alpha\gamma\alpha^{-1}\gamma^{-1}=\beta,  \alpha\beta\alpha^{-1}=\beta^{-1} \rangle.
 \end{equation}

A 3-torus is denoted by $\mathcal{G}_1$ in our paper. Note that the
fundamental group $\pi_{1}(\mathcal{G}_1)$ is represented in the
form

\begin{equation}\label{fundc1}
\pi_{1}(\mathcal{G}_1)=\ZZ^3
\end{equation}
It is also a crystallographic group $P1$.

\bigskip

\section{On the coverings of $\mathcal{B}_{1}$}\label{half1}
\subsection{The structure of the group $\pi_{1}(\mathcal{B}_{1})$ }

The following proposition provides the canonical form of an element
in $\pi_{1}(\mathcal{B}_{1})$.

\begin{proposition}\label{prop2}
\begin{itemize}
\item[(i)] Each element of $\pi_{1}(\mathcal{B}_{1})$ can be represented in the canonical form $a^xb^yc^z$
for some integer $x,y,z$.
\item[(ii)] The product of two canonical forms is given by the
formula
\begin{equation}\label{multlaw1}
a^xb^yc^z \cdot
a^{x'}b^{y'}c^{z'}=a^{x+x'}b^{(-1)^{x'}y+y'}c^{z+z'}.
\end{equation}
\item[(iii)] The canonical epimorphism $\phi: \pi_{1}(\mathcal{B}_{1}) \to
\pi_{1}(\mathcal{B}_{1})/\langle b\rangle \cong \ZZ^2$, given by the
formula $a^xb^yc^z \to (x,z)$ is well-defined.
\item[(iv)] The representation in the canonical form $a^xb^yc^z$ for
each element is unique.
\end{itemize}
\end{proposition}

\begin{proof} The part (i) follows from the part (ii),  part (ii) can be verified
directly. The part (iii) is equivalent the fact that the subgroup
$\langle b\rangle$ is a normal subgroup of
$\pi_{1}(\mathcal{B}_{1})$, which fact immediately follows from the
representation \ref{fundplus3} of $\pi_1(\mathcal{B}_1)$ . For the
proof of the part (iv) consider arbitrary element
$\pi_{1}(\mathcal{B}_{1})$ and its an arbitrary representation of
this element in the canonical form. The values of $x$ and $z$ are
uniquely defined by (iii). Thus $y$ is also uniquely defined, since
$b$ is an element of finite order in $Pc$ otherwise.
\end{proof}

\begin{lemma}\label{prop3}
Let  $\Delta$ be a subgroup of index $n$ in $\pi_1(\mathcal{B}_1)$.
By $l(\Delta)$ denote the minimal positive integer, such that
$b^{l(\Delta)} \in \Delta$. Such an $l(\Delta)$ exists, and satisfy
the relation $\l(\Delta)\cdot[\pi_1(\mathcal{B}_1) :
\phi(\Delta)]=n$.
\end{lemma}
\begin{proof}
If $l(\Delta)$ does not exists then all elements $b, b^2, b^3,\dots$
belong to mutually different cosets of $\Delta$ in
$\pi_{1}(\mathcal{B}_{1})$, thus the index of $\Delta$ is infinite,
which is a contradiction.

Let $g_1, \dots g_k$ be such elements of $\pi_1(\mathcal{B}_{1})$,
that $\phi(g_1), \dots \phi(g_k)$ is a complete system of right
coset representatives of $\phi(\Delta)$ in $\ZZ^2$. Then
$\{g_ib^j\mid 1 \le i \le k, 0 \le j \le l(\Delta)-1\}$ is a
complete system of right coset representatives of $\Delta$ in
$\pi_1(\mathcal{B}_{1})$. Thus $\l(\Delta)\cdot[\pi_1(\mathcal{B}_1)
: \phi(\Delta)]=[\pi_1(\mathcal{B}_{1}) : \Delta]=n$. \bigskip
\end{proof}

The next proposition shows that the introduced above invariant
$\phi(\Delta)$ is sufficient to determine whether $\Delta$ is
abelian or not.

\begin{proposition}\label{prop4}
Let $\Delta$ be a subgroup of finite index in
$\pi_{1}(\mathcal{B}_{1})$. Then $\Delta$ is abelian if and only if
$\phi(\Delta) \leqslant \{(2x,z)| \,\,x,z \in \ZZ \}$. \footnote{In
other words, $\Delta$ is abelian iff $\Delta \leqslant \langle
a^2,b,c\rangle$}
\end{proposition}
\begin{proof}
Since
$$
a^{2x}b^yc^z \cdot a^{2x'}b^{y'}c^{z'} = a^{2x+2x'}b^{y+y'}c^{z+z'}
= a^{2x'}b^{y'}c^{z'} \cdot a^{2x}b^yc^z,
$$
the if part is obvious. The inequality $a^{2x+1}b^yc^z \cdot
b^{l(\Delta)} \neq b^{l(\Delta)} \cdot a^{2x+1}b^yc^z$ proves the
only if part. Indeed, by \ref{multlaw1} $a^{2x+1}b^yc^z \cdot
b^{l(\Delta)} =a^{2x+1}b^{y+l(\Delta)}c^z$ and $b^{l(\Delta)} \cdot
a^{2x+1}b^yc^z = a^{2x+1}b^{y-l(\Delta)}c^z$, this expressions are
not equal since the order of $b$ is infinite.
\end{proof}

As a corollary of \Cref{prop4} we obtain the value of
$s_{\ZZ^3,\pi_1(\mathcal{B}_1)}(n)$.

\begin{corollary}\label{prop5}
The number of subgroups of index $n$ in $\pi_1(\mathcal{B}_{1})$,
which are isomorphic to $\ZZ^3$, is given by the
formula:
$$
s_{\ZZ^3, \pi_1(\mathcal{B}_{1})}(n) = \sum_{2l \mid n}
\sigma_1(l)l.
$$
\end{corollary}

\begin{proof}
Let $\Delta$ be a subgroup of index $n$ in $\pi_1(\mathcal{B}_{1})$
isomorphic to $\ZZ^3$, then by \Cref{prop4} $\phi(\Delta) \leqslant
\{(2x,z)| \,\,x,z \in \ZZ \}$. Thus from the definition of $\phi$ it
follows that $\Delta \leqslant <a^2,b,c> \cong \ZZ^3$. Since
$[\pi_{1}(\mathcal{B}_{1}) : <a^2,b,c>]=2$, we have $[<a^2,b,c> :
\Delta]=n/2$. The number of subgroups in $\ZZ^3$ of index $n/2$ is
well known (see, for instance, \cite{LisM} Corollary 4.4) and equals
to $\displaystyle\sum_{k \mid \frac{n}{2}} \sigma_1(k)k$.
\end{proof} \medskip

The following lemmas are technical statements, needed to introduce
the important invariant $\nu$.

\begin{lemma}\label{lemma2}
Let $\Delta$ be a subgroup of a finite index in
$\pi_1(\mathcal{B}_1)$. Then $\phi(\Delta)\cong \ZZ^2$.
\end{lemma}
\begin{proof}
The index $[\phi(\pi_1(\mathcal{B}_1)):\phi(\Delta)]$ divides
$[\pi_1(\mathcal{B}_1):\Delta]$, thus it is finite. Any subgroup of
finite index in $\ZZ^2$ is necessarily isomorphic to $\ZZ^2$.
\end{proof}

{\bf Notation.} By $\overline{v}=(x_v,z_v)$ and
$\overline{u}=(x_u,z_u)$ denote a pair of generators of
$\phi(\Delta)$, where $\phi(\Delta)$ is considered as a subgroup of
$\{(x,z)\mid x\in \ZZ, z\in \ZZ\}$.

\begin{lemma}\label{lemma3}
Any subgroup of finite index in $\langle a^2,b,c \rangle$ is
isomorphic to $\ZZ^3$.
\end{lemma}
\begin{proof}
The subgroup of finite index in $\ZZ^3$ is isomorphic to $\ZZ^3$.
\end{proof}

\begin{lemma}\label{prop7}
Let $(x,z)\in \phi(\Delta)$. Then there exist an integer number
$\mu(x,z),\,\, 0\le \mu(x,z) \le l(\Delta)-1$, such that for all
$a^xb^yc^z \in \Delta$ we have $y \equiv \mu(x,z) \mod l(\Delta)$.
\end{lemma}
\begin{proof}
Assume the converse is true. This means that there exist
$g=a^xb^yc^z \in \Delta$ and $h=a^xb^{y'}c^z \in \Delta$, such that
$y \not\equiv y' \mod l(\Delta)$. Since $h^{-1}g \in \Delta$,
$b^{y-y'} \in \Delta$, that contradicts the minimality of
$l(\Delta)$.
\end{proof}

\begin{lemma}\label{prop8}
Assume $\phi(\Delta) \nleqslant \{(2x,z)| \,\,x,z \in \ZZ \}$, then
one can choose the generators $\overline{v}=(x_v,z_v)$ and
$\overline{u}=(x_u,z_u)$ in such a way that $x_v$ is odd and $x_u$
is even.
\end{lemma}
\begin{proof}
At least one $x_v$ and $x_u$ is odd, otherwise the group $\Delta$,
generated by  $\overline{v}=(x_v,z_v)$ and $\overline{u}=(x_u,z_u)$
is a subgroup of $\{(2x,z)| \,\,x,z \in \ZZ\}$. Without loss of
generality suppose $x_v$ is odd. If $x_u$ is odd replace
$\overline{u}$ with $\overline{u}+\overline{v}$.
\end{proof}

From now on we fix some $\overline{v}$ and $\overline{u}$, chosen in
this way.

\bigskip

{\bf Notation.} Let $p,q \in \ZZ$, put
$(x,z)=p\overline{v}+q\overline{u}$. Denote $\nu(p,q)=\mu(x,z)$.

{\bf Remark.} Consider an arbitrary element of $\phi(\Delta)$,
$w=(x,z)=p\overline{v}+q\overline{u}$. Let
$a^{px_v+qx_u}b^yc^{pz_v+qz_u} \in \Delta$ be an arbitrary preiamge
of $w$ under $\phi$ . Then by definition $\nu(p,q) \equiv y
\mod{l(\Delta})$.

\begin{lemma}[almost additivity]\label{prop10}
$\nu(s+2p,t+q)\equiv \nu(s,t)+\nu(2p;q) \mod l(\Delta)$.
\end{lemma}
\begin{proof}
Let $g,h \in \Delta$ be a preimages of elements
$s\overline{v}+t\overline{u}$ and $2p\overline{v}+q\overline{u}$
respectively under the homomorphism $\phi$. In other words,
$g=a^{sx_v+tx_u}b^{\nu(s,t)+kl}c^{sz_v+tz_u}$ and
$h=a^{2px_v+qx_u}b^{k'l}c^{2pz_v+qz_u}$. Then
$gh=a^{(s+2p)x_v+(t+s)x_u}b^{\nu(s,t)+\nu(2p,q)+(k+k')l(\Delta)}c^{(s+2p)z_v+(t+s)z_u}$
by the formula \ref{multlaw1}. Thus $\nu(s+2p,t+q)\equiv
\nu(s,t)+\nu(2p;q) \mod l(\Delta)$ by the definition of
$\nu(s+2p,t+q)$.
\end{proof}
\begin{lemma}\label{prop9}
$\nu(2p,2q)=0$.
\end{lemma}
\begin{proof}
Let $g\in \Delta$ be a preimage of element $\overline{v}$ under the
homomorphism $\phi$. In other words, $g=a^{x_v}b^{y}c^{z_v}$. Then
$$
g^2=a^{2x_v}b^{(-1)^{x_v}y+y}c^{2z_v}=a^{2x_v}c^{2z_v}.
$$
The last equality holds since $x_v$ is odd. Thus $\nu(2,0)=0$.

Analogously, consider $h\in \Delta$ a preimage of
$\overline{v}+\overline{u}$ under $\phi$ to conclude $\nu(2,2)=0$.
Use \cref{prop10} to finish the proof.
%
\end{proof}

In other words to define the function $\nu(s,t)$ it is sufficient to
determine $\nu(0,0)$, $\nu(0,1)$, $\nu(1,0)$ and $\nu(1,1)$, also
\cref{prop9} gives $\nu(0,0)=0$.

\begin{lemma}\label{prop11}
The following holds:
\begin{itemize}
\item[(i)] $ \nu(1,1) \equiv \nu(0,1) + \nu(1,0) \mod{l(\Delta)}$
\item[(ii)] $ 2\nu(0,1) \equiv 0 \mod{l(\Delta)}$.
\end{itemize}
\end{lemma}

\begin{proof}
Immediate corollary of lemmas \ref{prop10} and \ref{prop9}.
\end{proof}

{\bf Notation.} Put $\rho(\Delta)=\nu(1,0)$ and
$\eps(\Delta)=\nu(0,1)$. \bigskip

{\bf Definition.} A 4-plet
$(l(\Delta),\phi(\Delta),\rho(\Delta),\eps(\Delta))$ is called {\em
$n$-essential} if the following conditions holds:
\begin{itemize}
\item[(i)] $l(\Delta)$ is a positive divisor of $n$,
\item[(ii)] $\phi(\Delta)$ is a subgroup of index $n/l(\Delta)$ in $\ZZ^2$, but not a subgroup of $\{(2p,q\mid p\in\ZZ, \,q\in\ZZ
)\}$,
\item[(iii)] $\rho(\Delta),\eps(\Delta) \in \{0,1,\dots ,l(\Delta)-1\}$, and $2\eps(\Delta) \equiv 0
\mod{l(\Delta)}$.
\end{itemize}

\begin{proposition}\label{prop12}
There is a bijection between the set of $n$-essential 4-plets
$(l(\Delta),\phi(\Delta),\rho(\Delta),\eps(\Delta))$ and non-abelian
subgroups $\Delta$ of index $n$ in $\pi_1(\mathcal{B}_1)$.
\end{proposition}

\begin{proof}
Let $\Delta$ be a non-abelian subgroup of index $n$ in
$\pi_1(\mathcal{B}_1)$. Since $(l(\Delta)$, $\phi(\Delta)$,
$\rho(\Delta)$ and $\eps(\Delta))$ are well-defined there is an
injection from the set of considered subgroups to the set of
4-plets. This 4-plets are $n$-essential in virtue of \cref{prop3}
and \cref{prop11}. Thus we have to show that every 4-plet is
achieved in this way.

Consider an $n$-essential 4-plet
$(l(\Delta),\phi(\Delta),\rho(\Delta),\eps(\Delta))$. Choose two
generating vectors $\overline{v}$ and $\overline{u}$ for the group
$\phi(\Delta)$ as described in \cref{prop8}. Direct verification
shows that the set
$$
\aligned & \Delta=\{a^{2px_v+2qx_u}b^{kl(\Delta}c^{2pz_v+2qz_u}|
\,\,p,q,k \in \ZZ\}\bigcup
\{a^{(2p+1)x_v+2qx_u}b^{\rho(\Delta+kl(\Delta}c^{(2p+1)z_v+2qz_u}|
\,\,p,q,k \in \ZZ\} \bigcup \\& \bigcup
\{a^{2px_v+(2q+1)x_u}b^{\eps(\Delta+kl(\Delta}c^{2pz_v+(2q+1)z_u}|
\,\,p,q,k \in \ZZ\} \bigcup\\&
\bigcup\{a^{(2p+1)x_v+(2q+1)x_u}b^{\rho(\Delta+\eps(\Delta+kl(\Delta}c^{(2p+1)z_v+(2q+1)z_u}|
\,\,p,q,k \in \ZZ\}
\endaligned
$$
is a subgroup in $\pi_{1}(\mathcal{B}_{1})$, again $\Delta$ have the
index $n$ in $\pi_1(\mathcal{B}_1)$ in virtue of \cref{prop3}.
\end{proof}

\begin{proposition}\label{prop13}
The type of nonabelian subgroup $\Delta$ of $\pi_1(\mathcal{B}_1)$
is uniquely determined by the value of $\eps(\Delta)$. More
precisely, if $\eps(\Delta) =0$ then $\Delta \cong Pc$, if
$l(\Delta)$ is even and $\eps(\Delta) =l(\Delta)/2$ then $\Delta
\cong Cc$.
\end{proposition}
\begin{proof}
In the case $\eps(\Delta) =0$ denote
$a'=a^{x_v}b^{\rho(\Delta)}c^{z_v}$, $b'=b^l(\Delta)$ and
$c'=a^{x_u}c^{z_u}$. Direct verification shows that the relations
$c'a'(c')^{-1}(a')^{-1}=c'b'(c')^{-1}b'^{-1}=e$ and
$a'b'(a')^{-1}b'=e$ holds. Further we call this relations {\em the
proper relations of the subgroup $\Delta$}. Thus the map $a \to a',
\,\, b \to b',\,\, c \to c'$ can be extended to an epimorphism
$\pi_1(\mathcal{B}_1) \to \Delta $. To prove that this epimorphism
is really an isomorphism we need to show that each relation in
$\Delta$ is a corollary of proper relations. We call a relation,
that is not a corollary of proper relations an {\em improper
relation}.

Assume the contrary. Since in $\Delta$ the proper relations holds,
each element can be represented in the canonical form, given by
\Cref{prop2} in terms of $a',b',c'$, by using just the proper
relations. I.e. each element $g$ can be represented as
$$
g=a'^{x}b'^{y}c'^{z}.
$$
If there is an improper relation then there is an equality
\begin{equation}\label{absurdum1}
a'^{p}b'^{q}c'^{r}=a'^{p'}b'^{q'}c'^{r'}
\end{equation}
where at least one of the inequalities $p\neq p'$, $q\neq q'$,
$r\neq r'$ holds. Substitute $a'=a^{x_v}b^{\rho(\Delta)}c^{z_v}$,
$b'=b^l(\Delta)$ and $c'=a^{x_u}c^{z_u}$ to \ref{absurdum1} and
apply the homomorphism $\phi$ to both left and right parts. We get
\begin{equation}\label{absurdum2}
\left\{
\begin{aligned}
px_v+rx_u= p'x_v+r'x_u\\
pz_v+rz_u= p'z_v+r'z_u \\
\end{aligned} \right.
\end{equation}
Since the vectors $\overline{v}=(x_v,z_v)$ and
$\overline{u}=(x_u,z_u)$ generate a subgroup of finite index in
$\ZZ^2$, the matrix $\begin{pmatrix} x_v & z_v \\ x_u &
z_u\end{pmatrix}$ is nonsingular. Thus \ref{absurdum2} implies
$p=p'$ and $r=r'$. That means $q\neq q'$.

So \ref{absurdum1} can be simplified to $b^{l(\Delta)(q-q')}=e$,
which is a contradiction since $\pi_1(\mathcal{B}_1)$ have no
elements of finite order.

In the second case denote $\alpha'=a^{x_v}b^{\rho}c^{z_v}$,
$\beta'=b^l$ and $\gamma'=a^{x_u}b^{-\eps}c^{z_u}$.  Direct
verification shows that the relations
$\alpha'\beta'(\alpha')^{-1}=(\beta')^{-1}$,
$\gamma'\beta'(\gamma')^{-1}=\beta'$ and
$\alpha'\gamma'(\alpha')^{-1}(\gamma')^{-1}=\beta'$ holds. Thus the
map $\alpha \to \alpha', \,\, \beta \to \beta',\,\, \gamma \to
\gamma'$ can be extended to an isomorphism $\pi_1(\mathcal{B}_1) \to
\Delta$. The proof is analogous to the previous case.
\end{proof}

{\bf Remark.} The groups $Pc$ and $Cc$ are not isomorphic since they
have different homologies: $H_1(Pc)=\ZZ_2\times\ZZ^2$ and
$H_1(Cc)=\ZZ^2$, see \cite{Wolf} or \cite{Conway}.

\subsection{The proof of Theorem 1}

Proceed to the proof of \Cref{th1}. First of all we show that there
exist only 3 types of subgroups in $\pi_{1}(\mathcal{B}_{1})$.
Consider a subgroup $\Delta$ of index $n$ in
$\pi_{1}(\mathcal{B}_{1})$. Then either $\phi(\Delta) \leqslant
\{(2p,q\mid p\in\ZZ, \,q\in\ZZ )\}$ or $\phi(\Delta) \nleqslant
\{(2p,q\mid p\in\ZZ, \,q\in\ZZ )\}$. In the first case the
\Cref{prop4} states that $\Delta$ is abelian and $\Delta \leqslant
\langle a^2,b,c \rangle$, thus  the group $\Delta \cong \ZZ^3$ as a
subgroup of finite index in $\ZZ^3$.

If $\phi(\Delta) \nleqslant \{(2p,q\mid p\in\ZZ, \,q\in\ZZ )\}$ then
$\Delta$ is bijectively determined by an $n$-essential 4-plet
$(l(\Delta),\phi(\Delta),\rho(\Delta),\eps(\Delta))$ in virtue of
\Cref{prop12}. Recall that $2\eps(\Delta) \equiv 0 \mod l(\Delta)$.
Thus there are only two cases: $\eps(\Delta)=0$ and
$\eps(\Delta)=l(\Delta)/2$ (the latter one is possible only if
$l(\Delta)$ is even).

In case $\eps(\Delta)=0$ \Cref{prop13} claims that $\Delta \cong
Pc$. In case $\eps(\Delta)=l(\Delta)/2$ \cref{prop13} yields $\Delta
\cong Cc$. Thus we proved that $\Delta$ is isimorphic to one of the
groups $Pc$, $\ZZ^3$ and $Cc$, and the latter two cases are possible
only if $n$ is even. Consider all three cases separately.

{\bf Case (i).} The number $s_{\ZZ^3, \pi_1(\mathcal{B}_{1})}(n)$ is
calculated in \Cref{prop5}.

{\bf Case (ii).} To find the number of subgroups, isomorphic to
$\pi_1(\mathcal{B}_{1})$ by Propositions \ref{prop12} and
\ref{prop13} we need to calculate the cardinality of the set of
$n$-essential 4-plets with $\eps(\Delta)=0$, i.e.
$$
\{(l(\Delta),\phi(\Delta),\rho(\Delta),0)\mid
(l(\Delta),\phi(\Delta),\rho(\Delta),0) \,\,\mbox{is an
$n$-essential 4-plet}\}.
$$
Keeping in mind the definition of an $n$-essential 4-plet we see
that $l(\Delta)$ is an arbitrary factor of $n$. The amount of
possible $\phi(\Delta)$ depending of $l(\Delta)$ may be calculated
the following way. By definition of $n$-essential 4-plet
$\phi(\Delta) \leqslant \ZZ^2, \phi(\Delta)\nleqslant \{(2p,q\mid
p\in\ZZ, \,q\in\ZZ )\}\,\, \mbox{and} \,\, \big[\ZZ^2 : \phi(\Delta)
\big]=n/l(\Delta)$. The total amount of $\phi(\Delta)$, such that
$\big[\ZZ^2 : \phi(\Delta) \big]=n/l(\Delta)$ is
$\sigma_1(\frac{n}{l(\Delta)})$, (see \cite{LisM} Corollary 4.4).
Analogously the amount of $\phi(\Delta)$, such that
$\phi(\Delta)\leqslant \{(2p,q\mid p\in\ZZ, \,q\in\ZZ )\}\,\,
\mbox{and} \,\, \big[\ZZ^2 : \phi(\Delta) \big]=n/l(\Delta)$ is
$\sigma_1(\frac{n}{2l(\Delta)})$. Thus amount of required
$\phi(\Delta)$ is
$\sigma_1(\frac{n}{l(\Delta)})-\sigma_1(\frac{n}{2l(\Delta)})$. The
amount of possible $\rho(\Delta)$ does not depends on a choice of
$\phi(\Delta)$ and equals $\l(\Delta)$. Thus for every fixed value
of $\l(\Delta)$ the amount of $n$-essential 4-plets with this
$\l(\Delta)$ and $\eps(\Delta)=0$ is
$\big(\sigma_1(\frac{n}{l(\Delta)})-\sigma_1(\frac{n}{2l(\Delta)})\big)l(\Delta)$.
Summing this amount over all possible values of $l(\Delta)$ we get
$$
\sum_{l \mid n}
(\sigma_1(\frac{n}{l})-\sigma_1(\frac{n}{2l}))l=\sum_{l \mid n}
\sigma_1(\frac{n}{l})l-\sum_{2l \mid n}\sigma_1(\frac{n}{2l})l.
$$

{\bf Case (iii).} Arguing similarly we get that the amount of
subgroups, isomorphic to $Cc$ is
$$
\sum_{2l \mid n} 2l\sigma_1(\frac{n}{2l})-\sum_{4l \mid
n}2l\sigma_1(\frac{n}{4l}).
$$

\subsection{The total number of subgroups of index $n$ in $\pi_1(\mathcal{B}_{1})$}\label{lulz}
As an immediate consequence of \Cref{th1} we get
$$
s_{\pi_1(\mathcal{B}_{1})}(n)=s_{\ZZ^3,
\pi_1(\mathcal{B}_{1})}(n)+s_{\pi_1(\mathcal{B}_{1}),
\pi_1(\mathcal{B}_{1})}(n) + s_{\pi_1(\mathcal{B}_{2}),
\pi_1(\mathcal{B}_{1})}(n).
$$
By the way there are at least two different considerations leading
to this result, we present them here.

\begin{proposition}

    \begin{equation}\label{mn1}
    s_{\pi_1(\mathcal{B}_{1})}(n) = \sum_{m\, |\, n} m \,c_{\pi_1(\mathcal{K})}(m) ,
        \end{equation}
        where
 \begin{equation}\label{um}
c_{\pi_1(\mathcal{K})}(m) = \left\{
\begin{array}{rcl}
\sigma_{0}(m), \,\text{if} \ \,  m \, is \,\, \text{odd}, \\
\frac{3}{2}\sigma_{0}(m) + \frac{1}{2} \sum_{d\, |\, \frac{m}{2}} (d-1), \,\, \text{if}\,\, \ m \,\, is \,\, \text{even}.\\
\end{array}
\right.
\end{equation}
\end{proposition}
\begin{proof}
In \cite{Stan} p.~112, in Equation 5.125, Stanley proves that if $G$
is a finitely generated group then \footnote{In Stanley notations:
$j_{G}(n)=c_{G}(n)$ and $u_{G}(n)=s_{G}(n)$.}

\begin{equation}\label{stanly}
s_{G \times \ZZ}(n) = \sum_{m\, |\, n} m \,c_{G}(m) .
\end{equation}
The formula \ref{um} is proven in \cite{MednKlein} (see Theorem 2).
Since $\pi_1(\mathcal{B}_{1})=\pi_1(\mathcal{K}) \times \ZZ$, to
finish the proof substitute \ref{um} to \ref{stanly}.
\end{proof}

 \bigskip

Some simple calculations, omitted here, show that the expressions
for $s_{ \pi_1(\mathcal{B}_{1})}(n)$, obtained by \Cref{th1} and by
\Cref{lulz} are equal.

A sequence $s_{\pi_1(\mathcal{B}_{1})}(n)$  coincides with the
sequence  A027844 in the  'On-Line Encyclopedia of Integer
Sequences' (\cite{Ency}).

\begin{remark} Note that the other formula for  $s_{\pi_1(\mathcal{B}_{1})}(n)$ was
obtained by different method by M.N.Shmatkov in PhD thesis
\cite{Shm} (see p.~150--151).
\end{remark}

\subsection{The proof Theorem 2}\label{lulz2}
To obtain the total number of $n$-coverings over  $\mathcal{B}_{1}$
we use following theorem from \cite{Medn}:
\begin{theorem*}[Mednykh]\label{teor_vse}
Let  $\Gamma$ be  a finitely generated group. Then the number of
conjugated classes of subgroups of index  $n$ in the \mbox{group
$\Gamma$,} is given by the formula
$$ c_{\Gamma}(n)=\frac{1}{n}\sum_{\substack{l | n \\ l m = n}}\, \sum_{K <_{m}\Gamma }\,|Epi(K,\ZZ_{l})|,$$
where the sum $  \sum_{K<_{m}\Gamma }$ is taken over all subgroups
$K$ of index $m$ in the group $\Gamma$ and $Epi(K,\ZZ_{l})$ is the
set of epimorphisms of the group $K$  onto the cyclic group
$\ZZ_{l}$ of order~$l$.
\end{theorem*}

Since \Cref{th1} classifies all subgroups of finite index in
$\pi_1(\mathcal{B}_{1})$, we just have to calculate
$|Epi(\ZZ^3,\ZZ_{l})|$, $|Epi(Pc,\ZZ_{l})|$ and $|Epi(Cc,\ZZ_{l})|$.

\begin{lemma}
\begin{itemize}\label{h1}
\item[(i)] $H_{1}(\ZZ^3, \ZZ) = \ZZ^{ 3 }$,

\item[(ii)] $H_{1}(Pc, \ZZ) =\ZZ_{2}\oplus\ZZ^{2}$,

\item[(iii)] $H_{1}(Cc, \ZZ) =\ZZ^{2}$,
\end{itemize}
where $H_{1}(\Delta, \ZZ)$ is a first homology group.
\end{lemma}

\begin{proof}
See \cite{Conway} section 7.
\end{proof}

The previous lemma and Lemma 4 of \mbox{\cite{Medn}} yield the
following result.
\begin{lemma} \label{epi}
We have
\begin{itemize}
\item[(i)]  $|Epi(\ZZ^3,\ZZ_{l})| = \sum_{\substack{d | l }}\,\mu \left( \frac{l}{d}\right)d^{3} : = \varphi_{3} (l)$,

\item[(ii)] $|Epi(Pc,\ZZ_{l})|=\sum_{\substack{ d | l}}\,\mu \left( \frac{l}{d}\right)(2,d) d^{2}$,

\item[(iii)]  $|Epi(Cc,\ZZ_{l})| = \sum_{\substack{d | l }}\,\mu \left( \frac{l}{d}\right)d^{2} : = \varphi_{2} (l)$,
\end{itemize}
 where  $\mu (n)$ is the M\"{o}bius function, $\varphi_{3} (l)$ and  $\varphi_{2} (l)$ are Jordan totient functions,$(2,d)$ is a greater common divisor of numbers  $2$ and $d$ .
\end{lemma}

Substituting the formulas from \Cref{epi} to Mednykh's Theorem we
get the statement of \Cref{th1.2}.

\subsection{The proof Theorem 3}

The isomorphism types of subgroups are already provided by
\Cref{th1}. Thus we have to calculate the number of conjugacy
classes for each type separately.

\begin{lemma}
$$
c_{\ZZ^3,\pi_{1}(\mathcal{B}_{1})}(n) = \frac{1}{2}\sum_{2l \mid n}
\sum_{m \mid \frac{n}{2l}} \Big(l^2 + \frac{5}{2}
+\frac{3}{2}(-1)^l\Big)m.
$$
\end{lemma}

\begin{proof}
Let $\Delta$ be a subgroup of index $n$ in $\pi_1(\mathcal{B}_{1})$,
isomorphic to $\ZZ^3$. Then $\Delta \leqslant \langle
a^2,b,c\rangle$ by \Cref{prop4}. Thus
$\Delta^{a^2}=\Delta^{b}=\Delta^{c}=\Delta$, so the conjugacy class
of $\Delta$ in $\pi_{1}(\mathcal{B}_{1})$ contains at most two
subgroups: $\Delta$ and $\Delta^a$. Thus we have to find out whether
$\Delta=\Delta^a$.

The arguments, analogous to \Cref{prop12} shows that there is a
bijection between the set of isomorphic to $\ZZ^3$ subgroups $\Delta
\leqslant \pi_{1}(\mathcal{B}_{1})$ and the setb 4-plets
$(l(\Delta),\phi(\Delta),y_v(\Delta),y_u(\Delta))$, such that
\begin{itemize}
\item[(i)]  $l(\Delta)$ is a positive divisor of $n$,
\item[(ii)] $\phi(\Delta) \leqslant \langle a^2,b,c \rangle$ and $[\langle a^2,b,c \rangle: \phi(\Delta)]=\frac{n}{2l(\Delta)}$,
\item[(iii)] $y_v(\Delta)$ and $y_u(\Delta)$ are arbitrary residues
modulo $l(\Delta)$.
\end{itemize}

Obviously, $l(\Delta)=l(\Delta^a)$, $\phi(\Delta)=\phi(\Delta^a)$,
$y_v(\Delta)=-y_v(\Delta^a)$, $y_u(\Delta)=-y_u(\Delta^a)$. Thus we
have to find the number of pairs $(y_v,y_u)$ of residues modulo
$l(\Delta)$, such that $y_v=-y_v$ and $y_u=-y_u$. If $l(\Delta)$ is
odd there is 1 pair $(0,0)$, if $l(\Delta)$ is even there are 4
pairs: $(0,0)$, $(0,l(\Delta)/2)$, $(l(\Delta)/2,0)$ and
$(l(\Delta)/2,l(\Delta)/2)$.
\end{proof}

\begin{lemma}\label{part2for-th3}
$$
c_{\pi_{1}(\mathcal{B}_{1}),\pi_{1}(\mathcal{B}_{1})}(n)=\sum_{l
\mid n} \Big(\frac{3}{2}+\frac{1}{2}(-1)^l\Big)\sigma_1(\frac{n}{l})
- \sum_{2l \mid n}
\Big(\frac{3}{2}+\frac{1}{2}(-1)^l\Big)\sigma_1(\frac{n}{2l}).
$$
\end{lemma}

\begin{proof}
Let $\Delta$ be a subgroup of index $n$ in $\pi_1(\mathcal{B}_{1})$
isomorphic to $\pi_1(\mathcal{B}_{1})$. Recall that by \Cref{prop12}
there is a bijection between the set of isomorphic to
$\pi_{1}(\mathcal{B}_{1})$ subgroups $\Delta \leqslant
\pi_{1}(\mathcal{B}_{1})$ and the set of $n$-essential 4-plets
$(l(\Delta),\phi(\Delta),\rho(\Delta),\eps(\Delta))$.

Obviously, $l(\Delta)=l(\Delta^d)$ and $\phi(\Delta)=\phi(\Delta^d)$
for any $d \in  \pi_{1}(+a_{1})$. Also
$\rho(\Delta)=-\rho(\Delta^a)$, $\rho(\Delta)=\rho(\Delta^b)+2$ and
$\rho(\Delta)=\rho(\Delta^c)$. Thus the parity of $\rho(\Delta)$ is
the only invariant for a conjugacy class with fixed $l(\Delta)$ and
$\phi(\Delta)$.

Summarizing the above considerations. $l(\Delta)$ can be any
positive divisor of $n$, thus we get the sum over all divisors the
amount of corresponding pairs $(\phi(\Delta),\rho(\Delta)\mod 2)$.
The amount of all $\phi(\Delta)$, such that
$[\ZZ^2:\phi(\Delta)]=n/l(\Delta)$ is
$\sigma_1(\frac{n}{l(\Delta)})$ the amount of $\phi(\Delta)$, such
that $\phi(\Delta)$ is a subgroup of index $\frac{n}{2l(\Delta)}$ in
$\{(2p,q\mid p\in\ZZ, \,q\in\ZZ )\}$ is
$\sigma_1(\frac{n}{2l(\Delta)})$ (again, we consider
$\sigma_1(\frac{n}{2l(\Delta)})=0$ if $\frac{n}{2l(\Delta)}$ is not
integer). Thus the amount $\phi(\Delta)$, satisfying the condition
of $n$-essential 4-plet is $\sigma_1(\frac{n}{l(\Delta)}) -
\sigma_1(\frac{n}{2l(\Delta)})$. We multiply it by the amoumt of
possible parities of $\rho(\Delta)$, i.e. 2 for even $l(\Delta)$ and
1 for odd, in other words by
$\Big(\frac{3}{2}+\frac{1}{2}(-1)^{l(\Delta)}\Big)$. Thus we get
$$
\aligned &
c_{\pi_{1}(\mathcal{B}_{1}),\pi_{1}(\mathcal{B}_{1})}(n)=\sum_{l
\mid n}
\Big(\frac{3}{2}+\frac{1}{2}(-1)^l\Big)\Big(\sigma_1(\frac{n}{l})-\sigma_1(\frac{n}{2l})\Big)=
\\& = \sum_{l
\mid n} \Big(\frac{3}{2}+\frac{1}{2}(-1)^l\Big)\sigma_1(\frac{n}{l})
- \sum_{2l \mid n}
\Big(\frac{3}{2}+\frac{1}{2}(-1)^l\Big)\sigma_1(\frac{n}{2l})
\endaligned
$$
\end{proof}

\begin{lemma}
$$
c_{\pi_{1}(\mathcal{B}_{2}),\pi_{1}(\mathcal{B}_{1})}(n) =
2\Big(\sum_{2k \mid n}\sigma_1(\frac{n}{2k}) - \sum_{4k \mid
n}\sigma_1(\frac{n}{4k})\Big).
$$
\end{lemma}

\begin{proof}
The proof is analogous to \Cref{part2for-th3}.
\end{proof}

\section{On the coverings of $\mathcal{B}_{2}$}
Most of the statements and proofs in this section are similar to
corresponding parts of section \ref{half1}. The proofs are given
only in case of significant difference.

\subsection{The structure of the group $\pi_{1}(\mathcal{B}_{2})$}

The following proposition provides the canonical form of an element
in $\pi_{1}(\mathcal{B}_{2})$.

\begin{proposition}\label{prop2-2}
\begin{itemize}
\item[(i)] Each element of $\pi_{1}(\mathcal{B}_{2})$ can be represented in the canonical form $\alpha^x\beta^y\gamma^z$
for some integer $x,y,z$.
\item[(ii)] The product of two canonical forms is given by the
formula
\begin{equation}\label{multlaw2}
\alpha^x\beta^y\gamma^z\cdot
\alpha^{x'}\beta^{y'}\gamma^{z'}=\left\{
\begin{aligned} \alpha^{x+x'}\beta^{y+y'}\gamma^{z+z'} \text{if } x' \text {is even} \\
\alpha^{x+x'}\beta^{-y-z+y'}\gamma^{z+z'} \text{if } x' \text {is odd} \\
\end{aligned} \right.
\end{equation}
\item[(iii)] The canonical epimorphism $\psi: \pi_{1}(\mathcal{B}_{2}) \to
 \pi_{1}(\mathcal{B}_{2})/\langle \beta \rangle \cong \ZZ^2$, given
by the formula $\alpha^x\beta^y\gamma^z \to (x,z)$ is well-defined.
\item[(iv)] The representation in the canonical form $\alpha^x\beta^y\gamma^z$ for
each element is unique.
\end{itemize}
\end{proposition}

\begin{lemma}\label{prop2-3}
Let  $\Delta$ be a subgroup of index $n$ in $\pi_1(\mathcal{B}_2)$.
By $l(\Delta)$ denote the minimal positive integer, such that
$\beta^{l(\Delta)} \in \Delta$. Such an $l(\Delta)$ exists, and
satisfy the relation $\l(\Delta)\cdot[\pi_1(\mathcal{B}_2) :
\psi(\Delta)]=n$.
\end{lemma}

The next proposition shows that the introduced above invariant
$\psi(\Delta)$ is sufficient to determine whether $\Delta$ is
abelian or not.

\begin{proposition}\label{prop2-4}
Let $\Delta$ be a subgroup of finite index in
$\pi_{1}(\mathcal{B}_{2})$. Then $\Delta$ is abelian if and only if
$\psi(\Delta) \leqslant \{(2x,z)| \,\,x,z \in \ZZ \}$. \footnote{In
other words, $\Delta$ is abelian iff $\Delta \leqslant \langle
\alpha^2,\beta,\gamma\rangle$}
\end{proposition}

As a corollary of \Cref{prop2-4} we obtain the value of
$s_{\ZZ^3,\pi_1(\mathcal{B}_2)}(n)$.

\begin{corollary}\label{prop2-5}
The number of subgroups of index $n$ in $\pi_1(\mathcal{B}_{2})$,
which are isomorphic to $\ZZ^3$, is given by the formula:
$$
s_{\ZZ^3, \pi_1(\mathcal{B}_{1})}(n) = \sum_{2l \mid n}
\sigma_1(l)l.
$$
\end{corollary}

The following lemmas are technical statements, needed to introduce
the important invariant $\nu$.

\begin{lemma}\label{lemma2-2}
Let $\Delta$ be a subgroup of a finite index in
$\pi_1(\mathcal{B}_2)$. Then $\psi(\Delta)\cong \ZZ^2$.
\end{lemma}

{\bf Notation.} By $\overline{v}=(x_v,z_v)$ and
$\overline{u}=(x_u,z_u)$ denote a pair of generators of
$\psi(\Delta)$, where $\psi(\Delta)$ is considered as a subgroup of
$\{(x,z)\mid x\in \ZZ, z\in \ZZ\}$.

\begin{lemma}\label{lemma2-3}
Any subgroup of finite index in $\langle \alpha^2,\beta,\gamma
\rangle$ is isomorphic to $\ZZ^3$.
\end{lemma}

\begin{lemma}\label{prop2-7}
Let $(x,z)\in \psi(\Delta)$. Then there exist an integer number
$\mu(x,z),\,\, 0\le \mu(x,z) \le l(\Delta)-1$, such that for all
$\alpha^x\beta^y\gamma^z \in \Delta$ we have $y \equiv \mu(x,z) \mod
l(\Delta)$.
\end{lemma}

\begin{lemma}\label{prop2-8}
Assume $\psi(\Delta) \nleqslant \{(2x,z)| \,\,x,z \in \ZZ \}$, then
one can choose the generators $\overline{v}=(x_v,z_v)$ and
$\overline{u}=(x_u,z_u)$ in such a way that $x_v$ is odd and $x_u$
is even.
\end{lemma}

From now on we fix some $\overline{v}$ and $\overline{u}$, chosen in
this way.

\bigskip

\begin{lemma}\label{prop2-8.5}
$z_u \equiv [\ZZ^2:\psi(\Delta)] \mod{2}$.
\end{lemma}
\begin{proof}
Follows from $[\ZZ^2:\psi(\Delta)] =|x_vz_u-x_uz_v|$, $x_v$ is odd
and $x_u$ is even.
\end{proof}

\begin{lemma}[almost additivity]\label{prop2-10}
$\nu(s+2p,t+q)\equiv \nu(s,t)+\nu(2p;q) \mod l(\Delta)$.
\end{lemma}

\begin{lemma}\label{prop2-9}
$\nu(2p,2q)=-pz_v-qz_u$.
\end{lemma}

\begin{proof}
Consider $g \in \Delta, \,\, \psi(g)=\overline{v}$, i.e.
$g=\alpha^{x_v}\beta^{\nu(1,0)+kl}\gamma^{z_v}$. Since $x_v$ is odd,
$$
g^2=\alpha^{2x_v}\beta^{-\nu(1,0)-kl-z_v+\nu(1,0)+kl}\gamma^{2z_v}=\alpha^{2x_v}\beta^{-z_v}\gamma^{2z_v}.
$$
Thus $\nu(2,0)=-z_v$. Use \Cref{prop2-10} to finish the proof. Also
note that in terms of $\mu$ the statement of \Cref{prop2-9} is much
shorter: $\mu(2x,2z)=-z$ if determined.
\end{proof}

\begin{lemma}\label{prop2-11}
The following holds:
\begin{itemize}
\item $ \nu(1,1) \equiv \nu(0,1) + \nu(1,0) \mod{l(\Delta)}$
\item $ 2\nu(0,1) \equiv \nu(0,2) \mod{l(\Delta)}$.
\end{itemize}
\end{lemma}

{\bf Notation.} Denote $\rho(\Delta)=\nu(1,0)$ and
$\eps(\Delta)=\nu(0,1)$.

Summing up Lemmas \ref{prop2-10}--\ref{prop2-11} we state the
following.
$$
\left\{
\begin{aligned}
\nu(2p,2q)\equiv -pz_v-qz_u \mod{l(\Delta)},\\
\nu(2p,2q+1)\equiv \nu(0,1)-pz_v-qz_u \mod{l(\Delta)},\\
\nu(2p+1,2q)\equiv \nu(1,0)-pz_v-qz_u \mod{l(\Delta)},\\
\nu(2p+1,2q+1)\equiv \nu(1,0)+\nu(0,1)-pz_v-qz_u \mod{l(\Delta)}.\\
\end{aligned}
\right.
$$

{\bf Definition.} A 4-plet
$(l(\Delta),\phi(\Delta),\rho(\Delta),\eps(\Delta))$ is called {\em
$n$-essential} if the following conditions holds:
\begin{itemize}
\item[(i)] $l(\Delta)$ is a positive divisor of $n$,
\item[(ii)] $\psi(\Delta)$ is a subgroup of index $n/l(\Delta)$ in $\ZZ^2$, but not a subgroup of $\{(2p,q\mid p\in\ZZ, \,q\in\ZZ
)\}$,
\item[(iii)] $\rho(\Delta),\eps(\Delta) \in \{0,1,\dots ,l(\Delta)-1\}$, and $2\eps(\Delta) \equiv
-z_u \mod{l(\Delta)}$.
\end{itemize}

\begin{proposition}\label{prop2-12}
There is a bijection between the set of $n$-essential 4-plets
$(l(\Delta),\phi(\Delta),\rho(\Delta),\eps(\Delta))$ and the set of
non-abelian subgroups $\Delta$ of index $n$ in
$\pi_1(\mathcal{B}_2)$.
\end{proposition}

\begin{proposition}\label{prop2-13}
The type of nonabelian subgroup $\Delta$ of $\pi_1(\mathcal{B}_2)$
is uniquely determined by the value of $\eps(\Delta)$. More
precisely, if $2\eps(\Delta) \equiv -z_u \mod{2l(\Delta)}$ then
$\Delta \cong \pi_1(\mathcal{B}_1)$,
 if $2\eps(\Delta)
\equiv -z_u+l(\Delta) \mod{2l(\Delta)}$ then $\Delta \cong
\pi_1(\mathcal{B}_2)$.
\end{proposition}

\subsection{The proof of Theorem 4}

Proceed to the proof of \Cref{th-Second-1}. First of all we show
that there exist only 3 types of subgroups in
$\pi_{1}(\mathcal{B}_{2})$. This proof follows the lines of the
proof of \Cref{th1}. Consider a subgroup $\Delta$ of index $n$ in
$\pi_{1}(\mathcal{B}_{2})$. Then either $\psi(\Delta) \leqslant
\{(2p,q\mid p\in\ZZ, \,q\in\ZZ )\}$ or $\phi(\Delta) \nleqslant
\{(2p,q\mid p\in\ZZ, \,q\in\ZZ )\}$. In the first case the
\Cref{prop2-4} states that $\Delta$ is abelian and $\Delta \leqslant
\langle a^2,b,c \rangle$, thus  the group $\Delta \cong \ZZ^3$ as a
subgroup of finite index in $\ZZ^3$.

If $\psi(\Delta) \nleqslant \{(2p,q\mid p\in\ZZ, \,q\in\ZZ )\}$ then
$\Delta$ is bijectively determined by an $n$-essential 4-plet
$(l(\Delta),\phi(\Delta),\rho(\Delta),\eps(\Delta))$ in virtue of
\Cref{prop2-12}. Recall that $2\eps(\Delta) \equiv -z_u \mod
l(\Delta)$. Thus there are only two cases: $2\eps(\Delta) \equiv
-z_u+l(\Delta) \mod 2l(\Delta)$ and $2\eps(\Delta)\equiv -z_u \mod
2l(\Delta)$  (the latter one is possible only if $l(\Delta)$ is
even).

In case $2\eps(\Delta) \equiv -z_u+l(\Delta) \mod 2l(\Delta)$
\Cref{prop2-13} claims that $\Delta \cong Cc$. In case
$2\eps(\Delta)\equiv -z_u \mod 2l(\Delta)$ \Cref{prop2-13} yields
$\Delta \cong Pc$. Thus we proved that $\Delta$ is isomorphic to one
of the groups $Cc$, $\ZZ^3$ and $Pc$, and the latter two cases are
possible only if $n$ is even. Consider all three cases separately.

{\bf Case (i).} The number $s_{\ZZ^3, \pi_1(\mathcal{B}_{2})}(n)$ is
calculated in \Cref{prop2-5}.

{\bf Case (ii).} To find the number of subgroups, isomorphic to
$\pi_1(\mathcal{B}_{2})$ by Propositions \ref{prop2-12} and
\ref{prop2-13} we need to calculate the cardinality of the set of
$n$-essential 4-plets with $2\eps(\Delta) \equiv -z_u+l(\Delta) \mod
2l(\Delta)$, i.e.
$$
\{(l(\Delta),\phi(\Delta),\rho(\Delta),0)\mid
(l(\Delta),\phi(\Delta),\rho(\Delta),\frac{-z_u+l(\Delta)}{2})
\,\,\mbox{is an n-essential 4-plet}\}.
$$
Keeping in mind the definition of an $n$-essential 4-plet we see
that $l(\Delta)$ is an arbitrary factor of $n$. The amount of
possible $\psi(\Delta)$ depending of $l(\Delta)$ may be calculated
the following way. By definition of $n$-essential 4-plet
$\psi(\Delta) \leqslant \ZZ^2, \psi(\Delta)\nleqslant \{(2p,q\mid
p\in\ZZ, \,q\in\ZZ )\}\,\, \mbox{and} \,\, \big[\ZZ^2 : \psi(\Delta)
\big]=n/l(\Delta)$. The total amount of $\psi(\Delta)$, such that
$\big[\ZZ^2 : \psi(\Delta) \big]=n/l(\Delta)$ is
$\sigma_1(\frac{n}{l(\Delta)})$, (see \cite{LisM} Corollary 4.4).
Analogously the amount of $\psi(\Delta)$, such that
$\psi(\Delta)\leqslant \{(2p,q\mid p\in\ZZ, \,q\in\ZZ )\}\,\,
\mbox{and} \,\, \big[\ZZ^2 : \psi(\Delta) \big]=n/l(\Delta)$ is
$\sigma_1(\frac{n}{2l(\Delta)})$ (we consider
$\sigma_1(\frac{n}{2l(\Delta)})=0$ if $\frac{n}{2l(\Delta)}$ is not
integer). Thus amount of required $\psi(\Delta)$ is
$\sigma_1(\frac{n}{l(\Delta)})-\sigma_1(\frac{n}{2l(\Delta)})$. The
amount of possible $\rho(\Delta)$ does not depends on a choice of
$\psi(\Delta)$ and equals $\l(\Delta)$. For every fixed 3-plet
$l(\Delta),\psi(\Delta),\rho(\Delta)$ either exists a unique
$\eps(\Delta)$, or none.

Unique $\eps(\Delta)$ exists if both $l(\Delta)$ and $z_u$ are odd,
or both are even, by \Cref{prop2-8.5} this means $l(\Delta)$ and
$\frac{n}{l(\Delta)}$ are both odd or both even. First case is
equivalent the statement that $n$ is odd, second case means
$\l(\Delta)=2k$ and $4k \mid n$ for some integer $k$. Thus we get
the formula.
$$
s_{\pi_{1}(\mathcal{B}_{2}),\pi_{1}(\mathcal{B}_{2})}=\left\{
\begin{aligned}
\sum_{4k \mid n}2k\big(\sigma_1(\frac{n}{2k})-\sigma_1(\frac{n}{4k})\big) \quad\text{if } n \,\,\text {is even} \\
\sum_{k \mid n} \sigma_1(\frac{n}{k})k \quad \text{if } n \,\,\text {is odd} \\
\end{aligned} \right.
$$

{\bf Case (iii).} Arguing similarly we get that the amount of
subgroups, isomorphic to $Cc$ is
$$
\sum_{2l \mid n} 2l\sigma_1(\frac{n}{2l})-\sum_{4l \mid
n}2l\sigma_1(\frac{n}{4l}).
$$

Thus the proof of \Cref{th-Second-1} is completed.

\begin{remark}
Notice that $s_{\pi_{1}(\mathcal{B}_{2})}$ was obtained by different
method by M.N.Shmatkov in PhD thesis \cite{Shm} (see p.~156--157)
and can be calculated by the following formula
\begin{equation}
s_{\pi_{1}(\mathcal{B}_{2})}= \left\{ \begin{aligned}
\sum_{k \mid n} \sigma_1(\frac{n}{k})k \quad \text{if } n \,\,\text {is odd} \\
\sum_{k \mid n,\, (2,\frac{n}{k})=(2,n) }((2,k)(\sigma_1(\frac{n}{k})-\sigma_1(\frac{n}{2k}))+k\sigma_1(\frac{n}{2k}))k \quad\text{if } n \,\,\text {is even} \\
\end{aligned} \right.
\end{equation}
\end{remark}

\subsection{The proof Theorem 5}\label{lulz2-2}

The proof of \Cref{th-Second-1.2} follows the same way as the proof
of \Cref{th1.2}.

Substituting the formulas from \Cref{th-Second-1}  and \Cref{epi} to
Mednykh's Theorem we get the statement of \Cref{th-Second-1.2}.

\subsection{The proof Theorem 6}

The isomorphism types of subgroups are already provided by
\Cref{th-Second-1}. Thus we have to calculate the number of
conjugacy classes for each type separately.

\begin{lemma}
$$
c_{\ZZ^3,\pi_1(\mathcal{B}_2)}(n)= \frac{1}{2}\sum_{2l \mid n}
\sum_{m \mid \frac{n}{2l}} \Big(l^2 +
\frac{3}{2}+\frac{1}{2}(-1)^{l}+(-1)^{\frac{n}{2lm}} +
(-1)^{l+\frac{n}{2lm}}\Big)m.
$$
\end{lemma}

\begin{proof}
Let $\Delta$ be a subgroup of index $n$ in $\pi_1(\mathcal{B}_{2})$,
isomorphic to $\ZZ^3$. Then $\Delta \leqslant \langle
\alpha^2,\beta,\gamma\rangle$ by \Cref{prop2-4}. Thus
$\Delta^{\alpha^2}=\Delta^{\beta}=\Delta^{\gamma}=\Delta$, so the
conjugacy class of $\Delta$ in $\pi_{1}(\mathcal{B}_{2})$ contains
at most two subgroups: $\Delta$ and $\Delta^{\alpha}$. Thus we have
to find out whether $\Delta=\Delta^{\alpha}$.

The arguments, analogous to \Cref{prop2-12} shows that there is a
bijection between an isomorphic to $\ZZ^3$ subgroups $\Delta
\leqslant \pi_{1}(\mathcal{B}_{2})$ and 4-plets
$l(\Delta),\psi(\Delta),y_v(\Delta),y_u(\Delta)$, such that
\begin{itemize}
\item $l(\Delta)$ is a positive divisor of $n$
\item $\phi(\Delta) \leqslant \langle a^2,b,c \rangle$ and $[\langle a^2,b,c \rangle: \phi(\Delta)]=\frac{n}{2l(\Delta)}$,
\item Choose two residues modulo $l(\Delta)$, $y_v(\Delta)$ and $y_u(\Delta)$. Here $\alpha^{2l_1}\beta^{y_v} \in \Delta$
and $\alpha^{2x_u}\beta^{y_u}\gamma^{z_u} \in \Delta$ are preimages
of $\overline{v}$ and $\overline{u}$ for the homomorphism $\psi$.
\end{itemize}

Obviously, $l(\Delta)=l(\Delta^{\alpha})$,
$\psi(\Delta)=\psi(\Delta^{\alpha})$,
$y_v(\Delta)=-y_v(\Delta^{\alpha})$,
$y_u(\Delta)=-y_u(\Delta^{\alpha})$. Thus we have to find the number
of pairs $(y_v,y_u)$ of residues modulo $l(\Delta)$, such that $y_v
\equiv -y_v$ and $y_u \equiv z_u -y_u \mod l(\Delta)$.
 If $l(\Delta)$ is odd there is 1 pair
$(0,\frac{z_u}{2})$,
 if $l(\Delta)$ is even and $\frac{n}{2l_1(\Delta)l(\Delta)}$ odd then there no such pairs, if both $l(\Delta)$ and $\frac{n}{2l_1(\Delta)l(\Delta)}$ are even then there are 4 pairs: $(0,\frac{z_u}{2})$, $(0,\frac{z_u+l(\Delta)}{2})$,
$(\frac{l(\Delta)}{2},\frac{z_u}{2})$ and
$(\frac{l(\Delta)}{2},\frac{z_u+l(\Delta)}{2})$. So the amount of
this pairs as a function of the parities of $l(\Delta)$ and
$\frac{n}{2l_1(\Delta)l(\Delta)}$ is given by the formula
$\frac{3}{2}+\frac{1}{2}(-1)^{l}+(-1)^{\frac{n}{2lm}} +
(-1)^{l+\frac{n}{2lm}}$.

\end{proof}

\begin{lemma}\label{part2for-th3}
$$
c_{\pi_1(\mathcal{B}_2),\pi_1(\mathcal{B}_2)}(n)=\left\{
\begin{aligned}
\sum_{4k \mid n}(\sigma_1(\frac{n}{2k})-\sigma_1(\frac{n}{4k})) \quad\text{if } n \,\,\text {is even} \\
\sum_{l \mid n} \sigma_1(\frac{n}{l}) \quad \text{if } n \,\,\text {is odd} \\
\end{aligned} \right. \leqno (ii)
$$
\end{lemma}

\begin{proof}
Let $\Delta$ be a subgroup of index $n$ in $\pi_1(\mathcal{B}_{2})$
isomorphic to $\pi_1(\mathcal{B}_{2})$. Recall that by
\Cref{prop2-12} and \Cref{prop2-13} there is a bijection between an
isomorphic to $\pi_{1}(\mathcal{B}_{2})$ subgroups $\Delta \leqslant
\pi_{1}(\mathcal{B}_{2})$ and $n$-essential 4-plets
$(l(\Delta),\phi(\Delta),\rho(\Delta),\eps(\Delta))$, such that
$2\eps(\Delta) \equiv -z_u+l(\Delta) \mod{2l(\Delta)}$.

Obviously, $l(\Delta)=l(\Delta^d)$ and $\phi(\Delta)=\phi(\Delta^d)$
for any $d \in  \pi_{1}(\mathcal{B}_{2})$.  Also
$\rho(\Delta^{\alpha})=-\rho(\Delta)+z_v$,
$\rho(\Delta^{\beta})=\rho(\Delta)+2$ and
$\rho(\Delta^{\gamma})=\rho(\Delta)+1$. Thus subgroups that differ
only in the parameter $\rho$ corresponds to one class of conjugancy.
Also the required $\eps(\Delta)$ exists iff $l(\Delta)$ is odd and
$z_v$ is odd or $l(\Delta)$ is even and $z_v$ is even. By
\cref{prop2-8.5} the $z_u \equiv \frac{n}{l(\Delta)} \mod2$, thus
$\eps(\Delta)$ exists if $n$ is odd or $l(\Delta) \mid \frac{n}{2}$
and $l(\Delta)$ is even. So, we get the amount
$$
\left\{ \begin{aligned}
\sum_{l \mid n} \sigma_1(\frac{n}{l}) \quad \text{if } n \,\,\text {is odd} \\
\sum_{l \mid \frac{n}{2},\, 2\mid l}(\sigma_1(\frac{n}{l})-\sigma_1(\frac{n}{2l})) \quad\text{if } n \,\,\text {is even} \\
\end{aligned} \right.
$$


Summarizing the above considerations. $l(\Delta)$ can be any
positive divisor of $n$, thus we get the sum over all divisors the
amount of corresponding pairs $(\phi(\Delta),\rho(\Delta)\mod 2)$.
The amount of all $\phi(\Delta)$, such that
$[\ZZ^2:\phi(\Delta)]=n/l(\Delta)$ is
$\sigma_1(\frac{n}{l(\Delta)})$ the amount of $\phi(\Delta)$, such
that $\phi(\Delta)$ is a subgroup of index $\frac{n}{2l(\Delta)}$ in
$\{(2p,q\mid p\in\ZZ, \,q\in\ZZ )\}$ is
$\sigma_1(\frac{n}{2l(\Delta)})$ (again, we consider
$\sigma_1(\frac{n}{2l(\Delta)})=0$ if $\frac{n}{2l(\Delta)}$ is not
integer). Thus the amount $\phi(\Delta)$, satisfying the condition
of $n$-essential 4-plet is $\sigma_1(\frac{n}{l(\Delta)}) -
\sigma_1(\frac{n}{2l(\Delta)})$. We multiply it by the amoumt of
possible parities of $\rho(\Delta)$, i.e. 2 for even $l(\Delta)$ and
1 for odd, in other words by
$\Big(\frac{3}{2}+\frac{1}{2}(-1)^{l(\Delta)}\Big)$. Thus we get
$$
\aligned &
c_{\pi_{1}(\mathcal{B}_{1}),\pi_{1}(\mathcal{B}_{1})}(n)=\sum_{l
\mid n}
\Big(\frac{3}{2}+\frac{1}{2}(-1)^l\Big)\Big(\sigma_1(\frac{n}{l})-\sigma_1(\frac{n}{2l})\Big)=
\\& = \sum_{l
\mid n} \Big(\frac{3}{2}+\frac{1}{2}(-1)^l\Big)\sigma_1(\frac{n}{l})
- \sum_{2l \mid n}
\Big(\frac{3}{2}+\frac{1}{2}(-1)^l\Big)\sigma_1(\frac{n}{2l})
\endaligned
$$
\end{proof}

\begin{lemma}
$$
c_{\pi_{1}(\mathcal{B}_{2}),\pi_{1}(\mathcal{B}_{1})}(n) =
2\Big(\sum_{2k \mid n}\sigma_1(\frac{n}{2k}) - \sum_{4k \mid
n}\sigma_1(\frac{n}{4k})\Big).
$$
\end{lemma}

\begin{proof}
The proof is analogous to \Cref{part2for-th3}.
\end{proof}

\section{Acknowledgment}
The authors are grateful to V. A. Liskovets, R. Nedela, M.~Shmatkov for helpful discussions during the work of this paper.

\section{Appendix}
Here we put some tables which illustrated our formulas. Note that numerical results  for $n$ from $1$ to $9$ were previously obtained in \cite{babaika}.





\bigskip

\subsection{Total number $c_{\pi_1(\mathcal{B}_1)}(n)$ of $n$-coverings over the first amphicosm}
\begin{center}
\begin{tabular}{|c|clllllllllllllll|}
\hline
\textbf{n }&1&2&3&4&5&6&7&8&9&10&11&12&13&14&15&16 \\ \cline{1-17}
\textbf{$c_{\pi_1(\mathcal{B}_1)}(n)$} & 1  & 7 & 5 & 23 & 7  & 39  & 9  & 65  & 18  & 61 & 13 & 143  & 15& 87& 35 & 183\\
\hline
\end{tabular}

\end{center}
\begin{center}
\small{Table~2}
\end{center}
 \bigskip

\subsection{The number $c_{\ZZ^3,\pi_1(\mathcal{B}_1)}(2n)$ of 3-torus  $2n$-coverings over the first amphicosm }
\begin{center}
\begin{tabular}{|c|clllllllllllllll|}
\hline
\textbf{n }&1&2&3&4&5&6&7&8&9&10&11&12&13&14&15&16\\ \cline{1-17}
\textbf{$c_{\ZZ^3,\pi_1(\mathcal{B}_1)}(2n)$} & 1  & 7  & 9  & 29  &19  & 63  & 33 & 107 & 74 & 133 & 73 & 285 & 99& 231& 219& 393\\
\hline
\end{tabular}

\end{center}
\begin{center}

\small{Table~3}
\end{center}
 \bigskip

\subsection{The number $c_{Cc,\pi_1(\mathcal{B}_1)}(2n)$ of the second amphicosm  $2n$-coverings over the first amphicosm }
\begin{center}
\begin{tabular}{|c|clllllllllllllll|}
\hline
\textbf{n }&1&2&3&4&5&6&7&8&9&10&11&12&13&14&15&16 \\ \cline{1-17}
\textbf{$c_{Cc,\pi_1(\mathcal{B}_1)}(2n)$} & 2  & 6  & 10 & 14  & 14 & 30  & 18 & 30 & 36 & 42 & 26 & 70 & 30& 54& 70& 62\\
\hline
\end{tabular}

\bigskip

\small{Table~4}
\end{center}

\subsection{The number $c_{Pc,\pi_1(\mathcal{B}_1)}(n)$ of the first amphicosm  $n$-coverings over the first amphicosm }
\begin{center}
\begin{tabular}{|c|clllllllllllllll|}
\hline
\textbf{n }&1&2&3&4&5&6&7&8&9&10&11&12&13&14&15&16 \\ \cline{1-17}
\textbf{$c_{Pc,\pi_1(\mathcal{B}_1)}(n)$} & 1  & 4  & 5 & 10  & 7 & 20  & 9 & 22 & 18 & 28 & 13 & 50 & 15& 36& 35& 46\\
\hline
\end{tabular}

\bigskip

\small{Table~5}
\end{center}


\subsection{Total number $c_{\pi_1(\mathcal{B}_2)}(n)$ of $n$-coverings over the second amphicosm}
\begin{center}
\begin{tabular}{|c|clllllllllllllll|}
\hline
\textbf{n }&1&2&3&4&5&6&7&8&9&10&11&12&13&14&15&16 \\ \cline{1-17}
\textbf{$c_{\pi_1(\mathcal{B}_2)}(n)$} & 1  & 3 & 5 & 13 & 7  & 19  & 9  & 43  & 18  & 33 & 13 & 93  & 15& 51& 35 & 137\\
\hline
\end{tabular}

\end{center}

\begin{center}
\small{Table~6}
\end{center}
 \bigskip

 \subsection{The number $c_{\ZZ^3,\pi_1(\mathcal{B}_2)}(2n)$ of 3-torus  $2n$-coverings over the second amphicosm }

\begin{center}
\begin{tabular}{|c|clllllllllllllll|}
\hline
\textbf{n }&1&2&3&4&5&6&7&8&9&10&11&12&13&14&15&16\\ \cline{1-17}
\textbf{$c_{\ZZ^3,\pi_1(\mathcal{B}_2)}(2n)$} & 1  & 5  & 9  & 23  &19  & 53  & 33 & 93 & 74 & 119 & 73 & 255 & 99& 213& 219& 363\\
\hline
\end{tabular}

\end{center}

\begin{center}

\small{Table~7}
\end{center}
 \bigskip

\subsection{The number $c_{Pc,\pi_1(\mathcal{B}_2)}(2n)$ of the first amphicosm  $2n$-coverings over the second amphicosm }

\begin{center}
\begin{tabular}{|c|clllllllllllllll|}
\hline
\textbf{n }&1&2&3&4&5&6&7&8&9&10&11&12&13&14&15&16 \\ \cline{1-17}
\textbf{$c_{Pc,\pi_1(\mathcal{B}_2)}(2n)$} & 2  & 6  & 10 & 14  & 14 & 30  & 18 & 30 & 36 & 42 & 26 & 70 & 30& 54& 70& 62\\
\hline
\end{tabular}

\end{center}

\begin{center}

\small{Table~8}
\end{center}

\subsection{The number $c_{Cc,\pi_1(\mathcal{B}_2)}(n)$ of the second amphicosm  $n$-coverings over the second amphicosm }
\begin{center}
\begin{tabular}{|c|clllllllllllllll|}
\hline
\textbf{n }&1&2&3&4&5&6&7&8&9&10&11&12&13&14&15&16 \\ \cline{1-17}
\textbf{$c_{Cc,\pi_1(\mathcal{B}_2)}(n)$} & 1  & 0  & 5 & 2  & 7 & 0  & 9 & 6 & 18 & 0 & 13 & 10 & 15& 0& 35& 14\\
\hline
\end{tabular}

\end{center}

\begin{center}

\small{Table~8}
\end{center}
\newpage

\label{finish}


\begin{thebibliography}{6}
\setlength{\itemsep}{0mm} \setlength{\labelwidth}{15mm}

\bibitem{Conway}
Conway~J.~H.  and  Rossetti~J.~P., Describing the Platycosms // arXiv:math.DG/0311476v1.

\bibitem{Hatcher}
Hatcher, Allen (2002), Algebraic Topology, Cambridge: Cambridge University Press, ISBN 0-521-79540-0.

\bibitem{Ha49} Hall~M., Jr., Subgroups of finite index in free
groups. {\it Canad. J. Math.}, {\bf 1}, No~1 (1949), 187--190.


\bibitem{Li71} V.~A.~Liskovets, Towards the enumeration
of subgroups of the free group. {\it Dokl. AN BSSR}, {\bf 15}, No~1
(1971), 6--9 (in Russian).


\bibitem{LisM}
Liskovets~V., Mednykh~A.,
Enumeration of subgroups in the fundamental groups of orientable circle bundles over surfaces, Commun. Algebra, \textbf{28}, No. 4 (2000), 1717--1738.

\bibitem{LSW}
Luminet~J.~P., Starkman~G.~D., Weeks~J.~R., Is Space Finite? // Scientific American
(1999), no. April, 90–97.


\bibitem{Med78} Mednykh~A.~D., Determination of the number of
nonequivalent co\-verings over a compact Riemann surface.
{\it Soviet Math. Dokl.}, {\bf 19}, No~2 (1978), \mbox{318--320}.
%

\bibitem{Med79} Mednykh~A.~D., On unramified co\-verings of
compact Riemann surfaces. {\it Soviet Math. Dokl.}, {\bf 20}, No~1 (1979),
85--88.

\bibitem{MednKlein}
Mednykh~A.~D., On the number of subgroups in the fundamental group of a closed
surfaces, Comm. in Algebra, \textbf{16}, No.10, (1988), 2137--2148.

\bibitem{MP86} Mednykh~A.~D. and Pozdnyakova~G.~G., Number
of nonequivalent coverings over a nonorientable compact surface.
{\it Siber.\! Math. J.}, {\bf 27}, No~1\,(1986),\! {99--106}.


\bibitem{Medn}
 Mednykh~A., Counting conjugacy classes of subgroups in a finitely generated group //Journal of Algebra. – 2008. – V. 320. – No. 6. – P. 2209--2217.

\bibitem{Ency}
On-Line Encyclopedia of Integer Sequences,
http://oeis.org/

\bibitem{Shm}
Shmatkov~M.~N, Enumeration of coverings over 3-manifolds. Ph.D.Thesis (in Russian), Sobolev Institute of Mathematics, Novosibirsk, 2004.

\bibitem{Stan}
Stanley~R.~J., Enumerative Combinatorics, vol. 2, Cambridge University Press, Cambridge,
1999.

\bibitem{babaika}
The Bilbao Crystallographic Server (http://www.cryst.ehu.es/cryst/cellsub.html)

\bibitem{We2}
Weeks~J.~R., Real-Time Rendering in Curved Spaces.// IEEE Computer Graphics and
Applications 22 (2002), no 6, 90–99.



\bibitem{Wolf}
 Wolf~J.~A., Spaces of Constant Curvature, Mc Graw-Hill Book Comp., New York, 1972.





















\end{thebibliography}
\end{document}